\documentclass[a4paper,10pt]{amsart}
\usepackage{etex}

\usepackage[latin1]{inputenc}
\usepackage{amsmath, amsthm, amssymb, comment}
\usepackage{amscd}
\usepackage[dvips]{graphicx}
\usepackage[all]{xy}
\usepackage{enumitem}
\usepackage{hyperref}
\hypersetup{
    colorlinks,
    linkcolor={red!50!black},
    citecolor={blue!50!black},
    urlcolor={blue!80!black}
}

\usepackage{caption}
\usepackage{subcaption}
\usepackage[rightcaption]{sidecap}

\usepackage[usenames,dvipsnames]{pstricks}
\usepackage{epsfig}
\usepackage{pst-grad} 
\usepackage{pst-plot} 
\usepackage{pstricks-add}
\usepackage{pst-solides3d}

\usepackage[margin=2.5cm]{geometry}

\usepackage{color}

\newtheorem{theorem}{Theorem}[section]
\newtheorem{corollary}[theorem]{Corollary}

\newtheorem{definition}[theorem]{Definition}
\newtheorem{lemma}[theorem]{Lemma}

\newtheorem*{theorem*}{Theorem}
\newtheorem*{proposition*}{Proposition}
\newtheorem*{definition*}{Definition}
\newtheorem*{lemma*}{Lemma}
\newtheorem*{claim*}{Claim}
\newtheorem*{corollary*}{Corollary}

\theoremstyle{definition}

\theoremstyle{remark}
\newtheorem{rem}[theorem]{Remark}
\newtheorem*{rem*}{Remark}

\newcommand{\wt}[1]{\widetilde{#1}}

\newcommand\R{\mathbb R}
\newcommand\F{\mathbb F}
\newcommand\Z{\mathbb Z}

\DeclareMathOperator{\id}{Id}

\newcommand\flot{ g^{t} }
\newcommand\hflot{ \tilde{ g}^t }

\newcommand\orb{ \mathcal O }
\newcommand\leafs{ \mathcal L ^{s} }

\newcommand\fs{\mathcal F^{s} }
\newcommand\hfs{\widetilde{\mathcal F}^{s} }
\newcommand\fu{\mathcal F^{u} }
\newcommand\hfu{\widetilde{\mathcal F}^{u} }

\newcommand{\FH}{\mathcal{FH}}

\title[Self orbit equivalences of Anosov flows]{A note on self orbit equivalences of Anosov flows and bundles with fiberwise Anosov flows}
\author{Thomas Barthelm\'e}
\address{Queen's University, Kingston, Ontario}
\email{thomas.barthelme@queensu.ca}
\urladdr{sites.google.com/site/thomasbarthelme}

\author{Andrey Gogolev}
 \thanks{The second author was partially supported by Simons grant 427063.}
\address{Binghamton University (SUNY), Binghamton, NY}
\curraddr{Ohio State University, Columbus, OH}
\email{gogolyev.1@osu.edu}

\begin{document}

\begin{abstract}
We show that a self orbit equivalence of a transitive Anosov flow on a $3$-manifold which is homotopic to identity has to either preserve every orbit or the Anosov flow is $\R$-covered and the orbit equivalence has to be of a specific type.
This result shows that one can remove a relatively unnatural assumption in a result of Farrell and Gogolev \cite{FG16} about the topological rigidity of bundles supporting a fiberwise Anosov flow when the fiber is $3$-dimensional.  
\end{abstract}

\maketitle

\section{Introduction}

In this note we give a complete description of self orbit equivalences of Anosov flows in dimension~$3$ that are homotopic to identity. Before stating the theorem, recall that an Anosov flow is called $\R$-covered if the leaf space of its weak stable (or equivalently weak unstable) foliation is homeomorphic to $\R$ (in particular, suspensions and geodesic flows are $\R$-covered).

\begin{theorem}\label{prop_orbit_eq}
 Let $M$ be a $3$-manifold and $\flot \colon M \rightarrow M$ be a transitive Anosov flow. Let $h \colon M \rightarrow M$ be a self orbit equivalence of $\flot$ which is homotopic to identity. Then we have:
\begin{enumerate}
 \item\label{it_notR_cov} If $\flot$ is not $\R$-covered, then $h$ preserves every orbit of $\flot$.
 \item\label{it_suspension_or_geod} If $\flot$ is orbit equivalent to a suspension of an Anosov diffeomorphism or the geodesic flow of a hyperbolic surface or orbifold, then $h$ preserves every orbit of $\flot$.
 \item\label{it_notTransversely_orient} If $\flot$ is $\R$-covered and its weak stable or unstable foliation is \emph{not} transversely orientable, then $h$ preserves every orbit of $\flot$.
 \item\label{it_Rcovered_trans} If $\flot$ is $\R$-covered but not one of the above cases, then there exist $\eta \colon M \rightarrow M$ (independent of $h$) a self orbit equivalence of $\flot$ which is homotopic to identity, $\eta \neq \id$, and $k \in \Z$ (depending on $h$) such that $h \circ \eta^k$ preserves every orbit of $\flot$. Moreover, there are two possible sub-cases
 \begin{enumerate}
  \item\label{subitem_infinite_order} Either, for all $i$, $\eta^i \neq \id$, in which case $k$ above is unique,
  \item\label{subitem_finite_order} Or, there exists a $q>0$ such that $\eta^i\neq \id$, for all $i=1,\ldots, q-1$ and $\eta^{q} = \id$ in which case $\flot$ is orbit equivalent to the $q$-cover of a geodesic flow on a hyperbolic surface or orbifold. In this case, the number $k$ is unique modulo $q$.
 \end{enumerate}

\end{enumerate}
\end{theorem}

\begin{rem}
If $h$ preserves every orbit of $g^t$ then it follows that $h$ is isotopic to identity via a path of self orbit equivalences which preserve orbits as well, see~\cite[Lemma 4.11]{FG16}.
\end{rem}

We now explain how the above description can be used to improve a topological rigidity result of Farrell and Gogolev which we proceed to describe.

Recall that a smooth locally trivial fiber bundle $p\colon E\to X$ whose fiber $M$ is a closed smooth manifold is called {\it Anosov} if the total space $E$ can be equipped with a continuous flow which preserves the fibers $M_x=p^{-1}(x)$ and whose restrictions to the fibers $g_x^t\colon M_x\to M_x$, $x\in X$, are $C^\infty$ Anosov flows which depend continuously on $x$ in the $C^1$-topology.
Also recall that $p\colon E\to X$ is called {\it topologically trivial (homotopically trivial) } if there exist a continuous trivialization $r\colon E\to M$ such that the restrictions $r_x\colon M_x\to M$ are homeomorphisms (homotopy equivalences) for each $x\in X$.

In \cite{FG16}, Farrell and Gogolev prove, among other results, the following rigidity result.
\begin{theorem*}[Theorem 4.2 of \cite{FG16}]
 Let $X$ be a closed manifold or a finite simplicial complex and let $p\colon E \rightarrow X$ be a homotopically trivial Anosov bundle whose fiberwise flows $g_x^t \colon M_x \rightarrow M_x$, $x\in X$, satisfy the following conditions:
\begin{enumerate}[label=\alph*)]
 \item The flows $g_x^t \colon M_x \rightarrow M_x$ are transitive Anosov flows;
 \item \label{assumption2} The flows $g_x^t \colon M_x \rightarrow M_x$ do not have freely homotopic orbits.
\end{enumerate}
Then the bundle $p\colon E \rightarrow X$ is topologically trivial.
\end{theorem*}

We will show that, thanks to Theorem~\ref{prop_orbit_eq}, the assumption~\ref{assumption2} turns out not to be needed for Farrell-Gogolev's proof to go through, when the fibers are $3$-dimensional. 

\begin{corollary}[to Theorem~\ref{prop_orbit_eq}]\label{cor_FG}
 Theorem 4.2 of \cite{FG16} holds when the fiber is $3$-dimensional assuming only that the fiberwise Anosov flows are transitive.
\end{corollary}

The reason one would want to remove the assumption~\ref{assumption2} of Farrell-Gogolev's result is that it seems to be a quite unnatural condition, as it does not appear to be verified very often. Indeed, in dimension $3$, Barthelm\'e and Fenley \cite{BartFe15} show that lots of (or even ``most'') examples have at least some free homotopy class with \emph{infinitely} many distinct orbits. In fact a (conjecturally) complete list of types of $3$-dimensional Anosov flows that satisfy assumption~\ref{assumption2} is as follows:
\begin{enumerate}
 \item The suspensions of Anosov diffeomorphisms;
 \item The geodesic flow of negatively curved surfaces or orbifolds;
 \item The (generalized) Bonatti-Langevin examples (see \cite{BonattiLangevin,Barbot:generalized_BL}).
\end{enumerate}
Notice that here, as in \cite{FG16} (but contrarily to the convention chosen in \cite{BartFe15}), we consider free homotopy of orbits that preserves the natural orientation given by the flow. Thus, in the case of the geodesic flow for instance, the pairs of periodic orbits representing the same closed geodesic, but travelling in the opposite directions are not freely homotopic to each others.

As for higher dimensions, it seems difficult to conjecture what is the most common behavior, as there are very few (classes of) examples. The two classical examples are suspensions of Anosov diffeomorphisms, and geodesic flows on manifolds of negative curvature. These examples do not admit freely homotopic orbits. But there are also some, quite special, higher dimensional examples constructed by surgery, which have pairs of freely homotopic periodic orbits \cite{BBGR}.

In the proof of Theorem 4.2 of \cite{FG16}, they need the assumption~\ref{assumption2} in only one place: to show that a self orbit equivalence of the fiberwise Anosov has to fix every orbit. Hence, thanks to Theorem \ref{prop_orbit_eq}, if the flow in the fiber is not $\R$-covered or is in cases (\ref{it_suspension_or_geod}) or (\ref{it_notTransversely_orient}), then we can use Farrell--Gogolev's proof without modification to obtain Corollary \ref{cor_FG}. However, if the flow is in case (\ref{it_Rcovered_trans}) of Theorem~\ref{prop_orbit_eq}, then more work is required to deduce the Corollary. We explain the necessary modification of Farrell--Gogolev's proof using the fact that $\eta$ is isotopic to identity. This isotopy result follows from work of Gabai-Kazez and Calegari (see Theorem~\ref{cor:Calegari} in the last section).

\begin{rem}
Because the fiber of the bundle is a $3$-dimensional manifold $M$, the triviality of the bundle actually follows from homotopic triviality without the dynamical assumption, that is, without assuming the existence of fiberwise Anosov flow. Namely, denote by $\text{G}_0(M)$, $\text{Top}_0(M)$ and $\text{Diff}_0(M)$ the spaces of self-homotopy equivalences, homeomorphisms and diffeomorphisms, respectively, which are homotopic to identity (i.e., the connected components of $id_M$). If the base $X$ is a simplicial complex then a homotopically trivial bundle can be topologically (smoothly) trivialized provided that the inclusion $\text{Top}_0(M)\subset\text{G}_0(M)$ ($\text{Diff}_0(M)\subset\text{G}_0(M)$) is a (weak) homotopy equivalence. Indeed, then the construction of such a trivialization can be carried out in a straightforward way by inductively extending the trivializing map over the skeleta of $X$. 

When $M$ is Haken $\text{Top}_0(M)\subset\text{G}_0(M)$ is a homotopy equivalence by work of Hatcher~\cite{H76} and Ivanov~\cite{I76}. 
Further combining with Hatcher's proof of the Smale conjecture $\text{Diff}(\mathbb D^3, \partial)\simeq *$~\cite{H83} one deduces that $\text{Diff}_0(M)\subset\text{G}_0(M)$ is a weak homotopy equivalence from the work of Cerf~\cite{cer68}. When $M$ is hyperbolic, Gabai proved that $\text{Diff}_0(M)\simeq *$~\cite{G01}.\footnote{We would like to thank Tom Farrell for his help with these references.} Hence, when $M$ is Haken or hyperbolic the bundle is, in fact, smoothly trivial due to deep work in $3$-manifold topology without employing the Anosov assumption. 
(The case when $X$ is a closed topological manifold would follow as in the end of the proof of Corollary \ref{cor_FG} in Section \ref{sec:proof_for_R_covered}).

However our proof of topological triviality (using the Anosov assumption) is much simpler, as it only relies on the fact  that  $\text{Diff}_0(M)$ is path-connected (see~Theorem~\ref{cor:Calegari}) and is indifferent to the type of 3-manifold $M$.
\end{rem}

\begin{rem}
The assumption of transitivity in Theorem~\ref{prop_orbit_eq} is unnecessary, but makes the proof easier. Given that this assumption is crucial for Farrell-Gogolev's result, we also include it here for the sake of brevity.
\end{rem}

The authors are very grateful to the anonymous referee for a meticulous reading and many comments which improved our note.

\section{Proof of Theorem \ref{prop_orbit_eq}}

In order to keep this note short, we will refer to Section 2 of \cite{BartFe15} for all the necessary background information on Anosov flows in dimension $3$. Notice that there is one important difference of convention: in~\cite{BartFe15}, free homotopy between orbits refers to a free homotopy between the \emph{non-oriented} curves represented by the orbits, whereas in this present note, we do \emph{not} forget the orientation.

What will be essential for the proof is the notion of \emph{lozenge}, \emph{chain of lozenges} and \emph{string of lozenges}. Call $\orb$ the orbit space of the Anosov flow $\flot$, i.e., $\orb = \wt M / \hflot$. The space $\orb$ with its induced topology is homeomorphic to $\R^2$.
For any point (or set) $x \in M$, we write $\fs(x)$ and $\fu(x)$ for the weak stable and weak unstable leaf through $x$. The notation $\hfs$ and $\hfu$ refers to the lifted foliations.

\begin{definition}
Let $\alpha,\beta$ be two orbits in $\orb$ and let $A \subset \hfs(\alpha)$, $B \subset \hfu(\alpha)$, $C \subset \hfs(\beta)$ and $D \subset \hfu(\beta)$  be four half leaves satisfying:
\begin{itemize}
 \item For any weak stable leaf $\lambda^s$, we have $\lambda^s \cap B \neq \emptyset$ if and only if $\lambda^s \cap D\neq \emptyset$,
 \item For any weak unstable leaf $\lambda^u$, we have $\lambda^u \cap A \neq \emptyset$ if and only if $\lambda^u \cap C \neq \emptyset$,
 \item The half-leaf $A$ does not intersect $D$ and $B$ does not intersect $C$.
\end{itemize}

A \emph{lozenge $L$ with corners $\alpha$ and $\beta$} is the open subset of $\orb$ given by (see Figure \ref{fig:a_lozenge}):
\begin{equation*}
 L := \lbrace p \in \orb \mid \hfs(p) \cap B \neq \emptyset, \; \hfu(p) \cap A \neq \emptyset \rbrace.
\end{equation*}
The half-leaves $A,B,C$ and $D$ are called the \emph{sides}.
\end{definition}

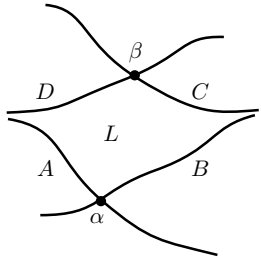
\begin{figure}[h]
\begin{center}
\scalebox{0.85}{
\begin{pspicture}(0,-1.97)(3.92,1.97)
\psbezier[linewidth=0.04](0.02,0.17)(0.88,0.03)(0.74114233,-0.47874346)(1.48,-1.11)(2.2188578,-1.7412565)(2.38,-1.73)(3.26,-1.95)
\psbezier[linewidth=0.04](0.0,0.27)(0.96,0.3105634)(0.75286174,0.37057108)(1.78,0.77)(2.8071382,1.169429)(2.66,1.47)(3.36,1.45)
\psbezier[linewidth=0.04](0.6,1.95)(1.2539726,1.8871263)(1.1265805,1.3646309)(2.0345206,0.7973154)(2.9424605,0.23)(3.2249315,0.2543258)(3.9,0.31)
\psbezier[linewidth=0.04](0.52,-1.33)(1.48,-1.33)(1.3597014,-0.9703507)(2.3,-0.63)(3.2402985,-0.28964934)(3.14,0.05)(3.84,0.23)
\psdots[dotsize=0.16](1.98,0.85)
\psdots[dotsize=0.16](1.46,-1.11)
\usefont{T1}{ptm}{m}{n}
\rput(1.6145313,-0.06){$L$}
\usefont{T1}{ptm}{m}{n}
\rput(1.4,-1.38){$\alpha$}
\rput(2,1.2){$\beta$}

\rput(0.6,-0.6){$A$}
\rput(3,-0.6){$B$}
\rput(0.6,0.6){$D$}
\rput(3,0.6){$C$}
\end{pspicture}
}
\end{center}
\caption{A lozenge with corners $\alpha$, $\beta$ and sides $A,B,C,D$} \label{fig:a_lozenge}
\end{figure}

We say that two lozenges are \emph{adjacent} if they either share a side (and a corner), or they share a corner.

\begin{definition} \label{def:chain_and_strings}
 A \emph{chain of lozenges} $\mathcal{C}$ is a finite or countable collection of lozenges such that, for any two lozenges $L, \bar L\in \mathcal{C}$, there exists a finite sequence of lozenges $L_0, \dots, L_n \in \mathcal{C}$ such that $L=L_0$, $\bar L = L_n$ and for any $i = 0, \dots n-1$, $L_i$ and $L_{i+1}$ are adjacent.

 A \emph{string of lozenges} is a chain of lozenges such that any two adjacent lozenges only share corners.
\end{definition}


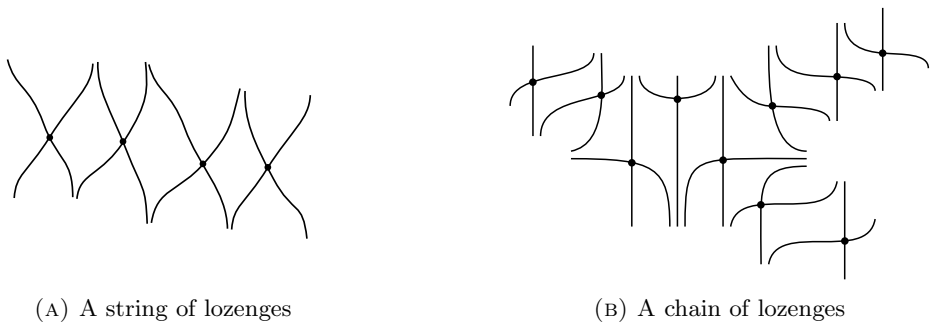
\begin{figure}[h]
\begin{subfigure}[b]{0.45\textwidth}
\centering
  \scalebox{0.55} { 
\begin{pspicture}(-3.8,-1)(4,4.4)
\rput{90}{
\psdots[dotsize=0.16](2.439,2.61)
\psbezier[linewidth=0.04](0.38,0.18)(1.24,0.04)(1.1011424,-0.46874347)(1.84,-1.1)(2.5788577,-1.7312565)(2.74,-1.72)(3.62,-1.94)
\psbezier[linewidth=0.04](0.36,0.28)(1.32,0.32056338)(1.1128618,0.38057107)(2.14,0.78)(3.1671383,1.1794289)(3.56,1.48)(4.26,1.46)
\psbezier[linewidth=0.04](0.96,1.96)(1.6139725,1.8971263)(1.4865805,1.3746309)(2.3945205,0.80731547)(3.3024607,0.24)(3.5849316,0.26432583)(4.26,0.32)
\psbezier[linewidth=0.04](0.24,-1.64)(1.16,-1.62)(1.7197014,-0.9603507)(2.66,-0.62)(3.6002986,-0.27964935)(3.5,0.06)(4.2,0.24)
\psdots[dotsize=0.16](2.34,0.86)
\psdots[dotsize=0.16](1.805,-1.05)
\psbezier[linewidth=0.04](0.98,3.46)(1.5739726,3.4171262)(1.5465806,3.1746309)(2.4545205,2.6073155)(3.3624606,2.04)(3.64,1.6)(4.22,1.58)
\psbezier[linewidth=0.04](1.0,2.06)(1.96,2.06)(1.8397014,2.4196494)(2.78,2.76)(3.7202985,3.1003506)(3.62,3.44)(4.32,3.62)
\psbezier[linewidth=0.04](0.22,-1.74)(0.8139726,-1.7828737)(0.7865805,-2.0253692)(1.6945206,-2.5926845)(2.6024606,-3.16)(2.88,-3.6)(3.46,-3.62)
\psbezier[linewidth=0.04](0.0,-3.54)(0.94,-3.4394367)(0.54,-3.22)(1.44,-2.74)(2.34,-2.26)(2.86,-2.04)(3.56,-2.06)
\psdots[dotsize=0.16](1.72,-2.605)
}
\end{pspicture}
}
  \caption{A string of lozenges}
\label{fig:string_lozenges} 
\end{subfigure}
\begin{subfigure}[b]{0.45\textwidth}
\centering
\scalebox{.5} 
{
\begin{pspicture}(0,-3.62)(11.02,3.62)
\psbezier[linewidth=0.04](6.8,2.6)(6.8,1.6)(6.8,-0.2)(7.8,-0.2)
\psline[linewidth=0.04cm](5.6,1.8)(5.6,-2.2)
\psbezier[linewidth=0.04](5.8,1.8)(6.0,1.4)(6.301685,1.0393515)(6.8,1.0)(7.298315,0.96064854)(8.4,1.0)(8.4,0.6)
\psbezier[linewidth=0.04](7.0,2.6)(7.0,2.0)(7.5016847,1.8393514)(8.0,1.8)(8.498315,1.7606486)(9.6,1.8)(9.6,1.4)
\psline[linewidth=0.04cm](8.6,0.6)(8.6,3.2)
\psbezier[linewidth=0.04](8.8,3.2)(8.8,2.6)(9.301684,2.4393516)(9.8,2.4)(10.298315,2.3606486)(11.0,2.4)(11.0,2.0)
\psline[linewidth=0.04cm](9.8,1.4)(9.8,3.6)
\psbezier[linewidth=0.04](7.8,-0.4)(4.8,-0.4)(4.6,-0.2)(4.6,-2.2)
\psline[linewidth=0.04cm](4.4,1.8)(4.4,-2.2)
\psbezier[linewidth=0.04](5.4,1.8)(5.4,1.0)(3.4,1.0)(3.4,1.8)
\psline[linewidth=0.04cm](3.2,1.8)(3.2,-2.2)
\psbezier[linewidth=0.04](4.2,-2.2)(4.2,-0.6)(3.8,-0.4)(1.6,-0.4)
\psbezier[linewidth=0.04](5.8,-2.2)(6.0,-1.0)(8.4,-2.2)(8.6,-1.0)
\psbezier[linewidth=0.04](7.8,-0.6)(6.4,-0.6)(6.6,-1.1777778)(6.6,-3.2)
\psbezier[linewidth=0.04](6.8,-3.2)(7.0,-2.0)(9.4,-3.2)(9.6,-2.0)
\psline[linewidth=0.04cm](8.8,-3.6)(8.8,-1.0)
\psbezier[linewidth=0.04](3.0,1.8)(3.0,1.0)(0.8,1.4)(0.8,0.2)
\psbezier[linewidth=0.04](1.6,-0.2)(2.6,0.0)(2.4,1.6)(2.4,2.4)
\psbezier[linewidth=0.04](2.2,2.4)(2.2,1.6)(0.0,2.0)(0.0,1.0)
\psline[linewidth=0.04cm](0.6,0.2)(0.6,2.6)
\psdots[dotsize=0.2](9.8,2.4)
\psdots[dotsize=0.2](8.6,1.78)
\psdots[dotsize=0.2](6.9,1)
\psdots[dotsize=0.2](5.6,-0.44)
\psdots[dotsize=0.2](6.595,-1.62)
\psdots[dotsize=0.2](8.8,-2.58)
\psdots[dotsize=0.2](4.4,1.18)
\psdots[dotsize=0.2](3.2,-0.51)
\psdots[dotsize=0.2](2.4,1.28)
\psdots[dotsize=0.2](0.6,1.63)
\end{pspicture} 
}
\caption{A chain of lozenges} 
\label{fig:chain_of_lozenges}
\end{subfigure}
\caption{Chain and string of lozenges} \label{fig:chain_and_string_lozenges}
\end{figure}


Lozenges are essential to the study of freely homotopic orbits because Fenley \cite{Fen:SBAF} proved that if a periodic orbit $\alpha$ is freely homotopic to another periodic orbit $\beta$ or its inverse then there exists two coherent lifts $\wt \alpha$ and $\wt \beta$ in $\wt M$ that are two corners of a chain of lozenges (see \cite[Section 2]{BartFe15}). Conversely, if there exists two lifts $\wt \alpha$ and $\wt \beta$ in $\wt M$ of periodic orbits that are two corners of a chain of lozenges, then, up to switching $\alpha$ and $\beta$, the orbit $\alpha$ is freely homotopic to $\beta^{\pm 1}$ or $\beta^{\pm 2}$.

Let $\wt \alpha$ be the lift of a periodic orbit, then one can show that there exists a uniquely defined maximal chain of lozenges which contains $\wt \alpha$ as a corner. We call this maximal chain $\FH(\wt \alpha)$ (see \cite[Section 2]{BartFe15}). By Fenley's result, the union of the corners in $\FH(\wt \alpha)$ contains (a coherent lift of) the complete free homotopy class of $\alpha$, as well as all the orbits freely homotopic to the inverse of $\alpha$.

We mention two results about lozenges that we will use:
\begin{lemma} \label{lem_no_corner_in_side_lozenges}
 Let $L_1$ and $L_2$ be two lozenges sharing a side. Then an orbit $\alpha$ in $L_1$ or $L_2$ is the corner of at most one lozenge.
\end{lemma}

\begin{proof}
Any orbit inside (the interior of) a lozenge is the corner of at most two lozenges (see \cite[Lemma 2.18]{BartFe15}). Now using the fact that $L_1$ and $L_2$ share a side one can see that one of the possible lozenges cannot exist.
\end{proof}

\begin{lemma}[\cite{BartFe15}] \label{lem_finitely_many_not_string}
 There are only finitely many periodic orbits $\alpha$ in $M$ such that $\FH(\wt \alpha)$ is \emph{not} a string of lozenges
\end{lemma}

\begin{proof}
 See Proposition 2.25 of \cite{BartFe15} and its proof.
\end{proof}

We finally recall a result, due to Barbot and Fenley, about $\R$-covered Anosov flows, which will give the map $\eta$ from case (\ref{it_Rcovered_trans}) of Theorem \ref{prop_orbit_eq}.
\begin{theorem}[Barbot \cite{Bar:CFA,Bar:PAG}, Fenley \cite{Fen:AFM}] \label{thm_Rcovered}
 Let $\flot$ be an $\R$-covered Anosov flow on a $3$-manifold $M$. Then
 \begin{itemize}
  \item Either $\flot$ is orbit equivalent to a suspension of an Anosov diffeomorphism;
  \item Or the weak stable (and the weak unstable) foliations of $\flot$ are \emph{not} transversely oriented;
  \item Or there exists a H\"older-continuous homeomorphism $\eta \colon M \rightarrow M$, homotopic to identity, such that $\eta^2$ is a self orbit equivalence of $\flot$. 
 
 Moreover, for any periodic orbit $\alpha$, its complete free homotopy class is given by $\{ \eta^{2i}(\alpha), i\in \Z \}$, and, for any lift $\wt \alpha$, we have
 \[
  \FH(\wt \alpha) = \{ \wt \eta^{i}(\alpha), i\in \Z \}.
 \]
 \end{itemize}
\end{theorem}

Since we will be using that fact later on, notice that, in the setting of Theorem \ref{thm_Rcovered}, all the orbits freely homotopic to a periodic orbit $\alpha$ are given by $\{ \eta^{2i}(\alpha), i\in \Z \}$, i.e., the orbit of $\alpha$ under \emph{even} powers of $\eta$, and all the orbits freely homotopic to \emph{the inverse} of $\alpha$ are given by the orbit of $\alpha$ under the \emph{odd} powers of $\eta$, i.e., $\{ \eta^{2i+1}(\alpha), i\in \Z \}$.

When an Anosov flow is $\R$-covered and not orbit equivalent to a suspension, Barbot \cite[Th\'eor\`eme C]{Bar:CFA} showed that the manifold supporting it is automatically orientable. However, the weak stable or weak unstable foliations might not be transversely orientable (and if one of them is not, since the manifold is orientable, the other is not either). If we fix an orientation of the, say, stable leaf space (which is orientable since it is homeomorphic to $\R$), for the stable foliation to be not transversely orientable means that there exists an element $\gamma \in \pi_1(M)$ that reverses the orientation (and hence must fix a stable leaf).
In this case, there still exists a homeomorphism $\wt\eta\colon \wt M \to \wt M$ in the universal cover that sends orbits to orbits of the lifted flow. However this $\wt\eta$ is now a \emph{twisted} $\pi_1(M)$-equivariant homeomorphism, i.e., for an orientation reversing $\gamma \in \pi_1(M)$, $\wt\eta \circ \gamma = \gamma \circ \wt\eta^{-1}$. So $\wt \eta$ does not descend to a homeomorphism of $M$.

To take care of the case where $\flot$ is $\R$-covered but not transversely orientable, we will use the following lemma.
\begin{lemma}\label{lem_not_transversely}
 If $\flot$ is $\R$-covered, not orbit equivalent to a suspension, and such that its weak stable or unstable foliations are not transversely orientable, then there exists a periodic orbit $\alpha$ that is freely homotopic to no other orbit.
\end{lemma}

\begin{proof}
 Consider $\gamma\in \pi_1(M)$ that reverses the orientation of the stable leaf space. The $\gamma$ fixes a stable leaf $L^s$, and thus represents a periodic orbit $\alpha$ of $\flot$.
 Let $\wt\alpha$ be the lift of $\alpha$ that is fixed by $\gamma$. Then, since $\gamma$ reverses the orientation of the stable leaf space, $\wt\alpha$ is the only corner fixed by $\gamma$ on the string of lozenge determined by $\wt\alpha$. So, in particular, no other orbit can be freely homotopic to $\alpha$ (but $\alpha^2$ is freely homotopic to other orbits or itself).
\end{proof}

\begin{proof}[Proof of Theorem \ref{prop_orbit_eq}]

We will start the proof by making several general remarks about the possible action of the lift of $h$ on the lifts of periodic orbits.

First, since $h$ is homotopic to identity, for any periodic orbit $\alpha$ of $\flot$, $h(\alpha)$ is freely homotopic to $\alpha$ (with the orientation preserved). We fix one homotopy of $h$ to identity and consider  the lift $\wt h$ of $h$ to $\wt M$ obtained via lifting that homotopy. Then the orbits $\wt\alpha$ and $\wt h(\wt\alpha)$ will either be equal, or be two corners on a chain of lozenges (by the above discussion, \cite{Fen:SBAF}).

Notice also that, as we constantly use it, that since $h$ is a self-orbit equivalence, it globally preserves the weak stable and weak unstable foliations.

Now suppose that $\alpha$ is a periodic orbit and $\wt \alpha$ is a lift. Let $\FH(\wt\alpha)$ be the maximal chain of lozenges containing $\wt \alpha$. Then $\wt h$ must preserve $\FH(\wt\alpha)$ globally. 

Suppose for a moment that $\wt h$ preserves one lozenge $L$ in $\FH(\wt\alpha)$, then we claim that $\wt h$ must act as the identity on the whole chain of lozenges $\FH(\wt\alpha)$.
Indeed, let $\beta_1$ and $\beta_2$ be the two corners of $L$. Since $L$ is preserved by $\wt h$, the corners are either fixed or switched. Let $g$ be the unique non-trivial element of $\pi_1(M)$ which fixes each corner of $\FH(\wt\alpha)$ and represents $\alpha$ (or $\alpha^2$ if the element representing $\alpha$ does not fix all the corners).

Now consider the action of $g$ on the weak stable and weak unstable leaves through $\beta_1$ and $\beta_2$ on the orbit space $\orb$. Since $\beta_1$ and $\beta_2$ are the two corner of a lozenge, one of two things can happen: Either $\beta_1$ is an attracting fixed point for the action of $g$ on (the projection to $\orb$ of) $\hfs(\beta_1)$, in which case $\beta_2$ must be a \emph{repelling} fixed point for the action of $g$ on (the projection to $\orb$ of) $\hfs(\beta_2)$. Or $\beta_1$ is a repelling fixed point for the action of $g$ on $\hfs(\beta_1)$ and $\beta_2$ must be an attracting fixed point for the action of $g$ on $\hfs(\beta_2)$. In either case the action of $g$ through the leaves of $\beta_1$ and $\beta_2$ in the orbit space $\orb$ are \emph{opposite}. Since $\wt h$ is a lift of $h$ obtained via an homotopy to identity, it commutes with all deck transformations, so, in particular, it commutes with $g$. Thus we conclude that $\wt h$ cannot switch $\beta_1$ and $\beta_2$. 
Therefore, $\wt h$ has to fix $\beta_1$ and $\beta_2$ as well as preserve each side of the stable and unstable leaves through $\beta_1$ and $\beta_2$. Therefore, $\wt h$ also has to preserve every other lozenge admitting $\beta_1$ or $\beta_2$ as a corner. Applying that argument inductively shows that $\wt h$ must then act as the identity on the whole chain of lozenges $\FH(\wt\alpha)$.

Hence we showed that, for any periodic orbit $\alpha$ and the corresponding chain of lozenges $\FH(\wt\alpha)$, either the map $\wt h$ does not preserve any lozenge in the chain, or, if it does preserve a lozenge, $\wt h$ must act as the identity on $\FH(\wt\alpha)$.

We can now head on to the proof.\\

\emph{Case 1:}
We start by assuming that $\flot$ is not $\R$-covered. Then there exists at least two periodic orbits $\alpha_1$ and $\alpha_2$, which lift to $\wt \alpha_1$ and $\wt \alpha_2$ in the universal cover $\wt M$ such that $\hfs(\wt\alpha_1)$ and $\hfs(\wt\alpha_2)$ are not separated (see \cite{Fen:SBAF}). Moreover, up to changing $\alpha_2$, we can assume that $\wt \alpha_1$ and $\wt \alpha_2$ are the opposite corners of two lozenges $L_1$ and $L_2$ that share a side (see Figure \ref{fig_for_proof}). Assume that $\wt \alpha_1$ is a corner of $L_1$, $\wt \alpha_2 $ a corner of $L_2$, and denote by $\wt \alpha_3$ the last corner. We want to show that $\wt h$ fixes $\wt \alpha_1$ and $\wt \alpha_2$, and hence the whole chain of lozenges containing them, i.e., we want to show that $\wt h$ acts as the identity on $\FH(\wt\alpha_1)$.

Let $B_1$ denote the quadrant determined by the stable and unstable leaves of $\wt \alpha_1$ opposite to $L_1$ (see Figure \ref{fig_for_proof}) and let $\hat B_1\supset B_1$ denote the union of $B_1$ and two quadrants adjacent to $B_1$, that is, $\hat B_1$ is the union of all quadrants at $\wt\alpha_1$, but the one containing $L_1$. Similarly, denote by $B_2$ the quadrant determined by the stable and unstable leaves of $\wt\alpha_2$ opposite to $L_2$ and by $\hat B_2$ the union of $B_2$ and two adjacent quadrants. Notice in particular that $\hat B_1 \cap \hat B_2  = \emptyset$.


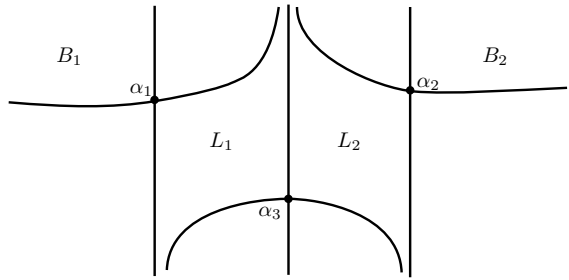
\begin{figure}[h]
\begin{center}
 \scalebox{0.8} 
{
\begin{pspicture}(0,-2.28)(9.2,2.28)
\psline[linewidth=0.04cm](2.4,2.24)(2.4,-2.24)
\psline[linewidth=0.04cm](6.6,2.22)(6.6,-2.26)
\psbezier[linewidth=0.04](2.6,-2.14)(2.64,-1.14)(4.040717,-0.9621294)(4.58,-0.96)(5.1192827,-0.9578706)(6.42,-1.2)(6.46,-2.18)
\psline[linewidth=0.04cm](4.6,2.26)(4.6,-2.22)
\psbezier[linewidth=0.04](4.44,2.22)(4.28,1.1)(3.78,1.02)(3.66,0.96)(3.54,0.9)(3.1255963,0.78)(2.408899,0.66)(1.6922019,0.54)(0.9555963,0.56)(0.0,0.64)
\psbezier[linewidth=0.04](4.74,2.22)(4.74,1.42)(6.069497,0.87071514)(6.66,0.82)(7.250503,0.7692849)(8.3,0.84)(9.18,0.88)
\psdots[dotsize=0.14](2.4,0.68)
\rput(2.18,0.88){$\wt\alpha_1$}
\psdots[dotsize=0.14](6.6,0.84)
\rput(6.9,0.99){$\wt\alpha_2$}
\psdots[dotsize=0.14](4.6,-0.96)
\rput(4.3,-1.24){$\wt\alpha_3$}
\rput(3.5,0){$L_1$}
\rput(5.6,0){$L_2$}
\rput(1,1.4){$B_1$}
\rput(8,1.4){$B_2$}
\end{pspicture} 
}
\caption{The lozenges $L_1$ and $L_2$} 
\label{fig_for_proof}
\end{center}
\end{figure}

Since periodic orbits of $g^t$ are dense and thanks to Lemma~\ref{lem_finitely_many_not_string} we can pick $\wt\beta_1\in L_1$ such that $\FH(\wt\beta_1)$ is a string of lozenges. Further, by Lemma~\ref{lem_no_corner_in_side_lozenges}, $\wt\beta_1$ must be an end-point of this string. Also note that all corners of 
$\FH(\wt\beta_1)$ except $\wt\beta_1$ must lie in $B_1$. If $\wt h(\wt\beta_1)=\wt\beta_1$ then $L_1$ is fixed. Hence, by the discussion preceding the start of \emph{Case 1}, $\wt h$ must act as the identity on the corners of the whole chain of lozenges $\FH(\alpha_1)$. This is what we wanted to show.
We thus assume for a contradiction that $\wt h(\wt\beta_1)\neq\wt\beta_1$ and, hence, $\wt h(\wt\beta_1)\in B_1$.
Analogously, either $\wt h$ acts as the identity on $\FH(\alpha_1)$ or we can find a periodic orbit $\wt\beta_2\in L_2$ such that $\wt h(\wt\beta_2)\in B_2$.

Because $L_1$ and $L_2$ share a side, by lozenges property, the leaf $\hfs(\wt \beta_1)$ intersects the leaf $\hfu(\wt \beta_2)$. On the other hand we have $\wt h(\hfs(\wt \beta_1))=\hfs(\wt h(\wt \beta_1))\subset \hat B_1$ because $\wt h(\wt \beta_1)\in B_1$. And similarly, $\wt h(\hfu(\wt \beta_2))=\hfu(\wt h(\wt \beta_2))\subset \hat B_2$. Because $\hat B_1$ and $\hat B_2$ are disjoint we obtain that $\wt h(\hfs(\wt \beta_1))$ and $\wt h(\hfu(\wt \beta_2))$ are disjoint which yields a contradiction.
Therefore $\wt h$ must fix $\wt \beta_1$, thus $L_1$, thus, as explained above, it acts as the identity on the maximal chain of lozenges containing $\wt \alpha_1$.

So far, we showed that any chain of lozenges such that some of its corners are on non-separated leaves must be fixed by $\wt h$. Now, take any periodic orbit $\alpha$ in $M$ and a lift $\wt \alpha$. We want to show that $\wt h(\wt \alpha) = \wt \alpha$. If $\FH(\wt \alpha)$ is not a string of lozenges, i.e., there exists two corners on non-separated leaves (see \cite[Section 2]{BartFe15}), then we already showed that $\wt \alpha$ is fixed. So we may assume that $\FH(\wt \alpha)$ is a string of lozenges.
Recall that every weak stable and weak unstable leaf is dense in $M$. So, there exists an element $\gamma\in \pi_1(M)$ such that the stable leaf of $\wt \alpha$ intersect the lozenge $\gamma\cdot L_1$.

Since $\wt h$ fixes $L_1$ (and thus $\gamma\cdot L_1$ as $\wt h$ commutes with the deck transformations), it must send $\wt \alpha$ to another corner of $\FH(\wt \alpha)$ such that its weak stable leaf also intersects $\gamma\cdot L_1$.

If there are no such other corners in $\FH(\wt \alpha)$, then we obtained $\wt h(\wt \alpha) = \wt \alpha$, as we desired. So suppose there exists $\wt \beta$ a corner of $\FH(\wt \alpha)$ such that $\hfs(\wt \beta)$ intersects $\gamma\cdot L_1$. Since $\FH(\wt \alpha)$ is a string of lozenges, such a corner $\wt \beta$ is necessarily unique and must be one of the opposite corners of $\wt \alpha$. Let $L$ be the lozenge with corner $\wt\alpha$ and $\wt\beta$. Then $L$ intersects $\gamma\cdot L_1$, which is fixed by $\wt h$, thus $L$ is fixed by $\wt h$, and this once more implies that $\wt h(\wt \alpha) = \wt \alpha$.

So every periodic orbit in $\orb$ has to be fixed by $\tilde h$, and by density of periodic orbits (since $\flot$ is transitive), every orbit is fixed. This concludes the proof of item 1 of the Theorem.\\

\emph{Case 2:}
Now, suppose that $\flot$ is $\R$-covered. Then, according to Theorem \ref{thm_Rcovered}, there are three possible cases. If $\flot$ is orbit equivalent to a suspension of an Anosov diffeomorphism, then there does not exist any distinct pair of periodic orbits that are freely homotopic and hence $h$ fixes every orbit (see, for instance, \cite{BartFe15} or \cite{FG16}). So we are in case (\ref{it_suspension_or_geod}), and we are done.

So let us assume that $\flot$ is $\R$-covered but not orbit equivalent to a suspension. Then $M$ is oriented (and transitivity is automatic) \cite{Bar:CFA}.
Notice that, given our assumption, every chain of lozenges is a string of lozenges.

Suppose first that there exists a periodic orbit $\alpha$ such that one of its lifts (and hence all) is fixed by $\wt h$ and a periodic orbit $\beta$ such that $h(\beta) \neq \beta$. Then, using density of weak stable and unstable leaves once again, one can choose lifts $\wt \alpha$ and $\wt \beta$ such that $\hfu(\wt \alpha)$ intersect $\hfs(\wt \beta)$. So $\hfs(\wt h(\wt \beta))$ must also intersects $\hfu(\wt \alpha)$, and, since $\flot$ is $\R$-covered, this implies that $\wt h(\wt \beta)$ must be one of the corners opposite to $\wt \beta$ in the string of lozenges $\FH(\wt\beta)$. Then, as before, since the action (in the orbit space) of the deck transformations that fixes all the corners of $\FH(\wt\beta)$ must be opposite on opposite corners of any lozenge and $\wt h$ commutes with deck transformations, we obtain a contradiction.

So either every lift of periodic orbits is fixed, or they are all moved by $\wt h$.

According to Lemma \ref{lem_not_transversely}, if one of the foliations of $\flot$ is not transversely orientable, then there exists at least one orbit fixed by $\wt h$, thus all orbits are fixed and we are in case (\ref{it_notTransversely_orient}) of Theorem \ref{prop_orbit_eq}.

Otherwise, Theorem \ref{thm_Rcovered} gives us a homeomorphism homotopic to identity $\eta$. Let $\wt \eta$ be the lift to the universal cover given by the homotopy to $\id$.

We suppose first that $\wt h$ moves all lifts of periodic orbits. Then, for every periodic orbit $\alpha$, since $\FH(\alpha) = \{ \eta^{2i}(\alpha), i\in \Z \}$ (by Theorem \ref{thm_Rcovered}), there exists $k_{\alpha}$ such that $\wt h(\wt\alpha) = \wt\eta^{ 2k_{\alpha}}(\wt\alpha)$. Our goal is to show that $k_{\alpha}$ is in fact independent of the orbit $\alpha$.

Let $\alpha$ be a fixed periodic orbit, and $\wt\alpha$ a lift. Let $\wt \beta$ be a lift of another periodic orbit. Up to acting by the fundamental group, we can assume that $\hfs(\wt\beta) \cap \hfu(\wt\alpha) \neq \emptyset$. Furthermore, still without loss of generality, we assume that $\hfs(\wt\beta)$ is above $\hfs(\wt\alpha)$ for the orientation that makes the projection of $\eta$ to the stable leaf space $\leafs$ an increasing function.

Hence, $\hfs(\wt h(\wt\beta))$ has to intersect $\hfu(\wt h(\wt\alpha)) = \hfu(\wt\eta^{2k_{\alpha}}(\wt\alpha))$, and has to be above $\hfs(\wt\eta^{2k_{\alpha}}(\wt\alpha))$, because $h$ preserves the orientation. So, in the leaf space $\leafs \simeq \R$, the stable leaves $\hfs(\wt h(\wt\beta))$ and $\hfs \left(\wt\eta^{2k_{\alpha}}(\wt\beta) \right)$ both have to be in between $\hfs(\wt\eta^{2k_{\alpha}}(\wt\alpha))$ and $\hfs(\wt\eta^{2k_{\alpha} + 1}(\wt\alpha))$. Adding the fact that $\wt h(\wt\beta)$ and $\wt\eta^{2k_{\alpha}}(\wt\beta)$ are both corners of the same string of lozenges, we conclude that $\wt h(\wt\beta) = \wt\eta^{2k_{\alpha}}(\wt\beta)$.

So $k_{\alpha} = k_{\beta}$ and since this is true for any $\beta$, we proved that there exists $k \neq 0$ depending only on $h$ and the flow, such that the action of $\wt h$ on the orbit space is the same as the action of $\wt \eta^{2k}$, i.e., $\wt h \circ \wt \eta^{-2k}$ preserves every orbit of the flow.

We obtained that $\wt h \circ \wt \eta^{-2k}$ preserves every orbit of the flow, for some $k\neq 0$, when we assumed that $\wt h$ moved every lifts of periodic orbits. The other possibility was that $\wt h$ fixes every lifts of periodic orbits, which implies that $\wt h \circ \wt \eta^{0}$ preserves every orbit of the flow. So, in either case, we have that $\wt h \circ \wt \eta^{-2k}$ preserves every orbit of the flow for some $k\in \Z$.

Now, using \cite[Theorem B]{BartFe15}, we conclude that either $\eta^{i} \neq \id$ for any $i$, or there exists a minimal, strictly positive integer, $q$ such that $\eta^{2q} = \id$, in which case, $\flot$ is orbit equivalent to a $q$-cover of the geodesic flow of a hyperbolic surface or orbifold. Renaming $\eta^2$ by $\eta$, and $k$ by $-k$ we get the exact statement of Theorem \ref{prop_orbit_eq}.
\end{proof}

\section{Proof of Corollary \ref{cor_FG} for $\R$-covered fiberwise flows} \label{sec:proof_for_R_covered}

We will need to use the following homotopy-implies-isotopy result of Calegari based on a homotopy-implies-isotopy result of Gabai-Kazez.

\begin{theorem}[Corollary~5.3.12 of~\cite{Cal00}, \cite{GK97}]\label{cor:Calegari}
If a 3-dimensional manifold $M$ admits an $\R$-covered, transversely orientable, foliation then any homeomorphism $\eta\colon M\to M$ homotopic to the identity is isotopic to the identity.
\end{theorem}

The only step in which assumption~\ref{assumption2} of Theorem 4.2~\cite{FG16} is used, is to show that if $h \colon M \rightarrow M$ is a self orbit equivalence of a transitive Anosov flow, homotopic to the identity, then $h$ preserves every orbit. 
Therefore, by Theorem~\ref{prop_orbit_eq}, to establish Corollary~\ref{cor_FG} we only have to deal with the case~(\ref{it_Rcovered_trans}): When the Anosov flow in the fiber is $\R$-covered with transversely oriented stable and unstable foliations and not a suspension flow of an Anosov diffeomorphism or a geodesic flow on a hyperbolic surface or orbifold. In particular, we can use Theorem \ref{cor:Calegari} in this case.

Recall that the proof of Theorem~4.2 of~\cite{FG16} is based on reducing the structure group of $p\colon E\to X$ from $\text{Top}(M)$ to a special subgroup $F<\text{Top}(M)$, which comes from using structural stability of Anosov flows. Recall that group $F$ consists of all homeomorphisms $\varphi\colon M\to M$ such that
\begin{enumerate}
\item[P1] $\varphi\colon M\to M$ is a self orbit equivalence of $g^t\colon M\to M$, where $g^t$ is the Anosov flow on a ``base-point" fiber $M=M_{x_0}$;
\item[P2] $\varphi$ is homotopic to identity $\id$;
\item[P3] $\varphi$ is bi-H\"older continuous;
\item[P4] $\varphi$ is differentiable along the orbits of the flow with a H\"older continuous derivative.
\end{enumerate}

Then by Theorem~\ref{prop_orbit_eq}, for each $\varphi\in F$ there exists a unique  $k\in\F$ such that $\varphi\circ \eta^k$ preserves the orbits of $g^t$. Here the group $\F$ is either $\Z$ or $\Z/q\Z$ according to whether we are in the case \ref{subitem_infinite_order} or \ref{subitem_finite_order} of Theorem~\ref{prop_orbit_eq}. Hence, we can decompose $F$ as the disjoint union
$$
F=\bigcup_{k\in\F}F_k,
$$
where $F_k$ consists of $\varphi\in F$ for which $\varphi\circ\eta^k$ preserves the orbits of $g^t$. It is well known that $\eta$ is H\"older continuous and can be chosen to be uniformly smooth along the orbits of $g^t$ \cite[Theorem 19.1.5]{KH95}. A calculus exercise can show that a H\"older continuous orbit equivalence which is uniformly smooth along the orbits also must satisfy P4 and, hence, $\eta\in F$. It follows that
\begin{equation}
\label{eq_Fk}
F_k=F_0\circ \eta^k
\end{equation}


Because elements of $F_0$ preserve orbits of $g^t$, Lemma~4.11 of~\cite{FG16} applies to $F_0$ and provides a deformation retraction of $F_0$ to $\id$. Then the same deformation retraction can be used to retract $F_k$ to $\eta^k$ according to~(\ref{eq_Fk}). We conclude that we have further reduced the structure group to $\{\eta^k; k\in \F\}< F$, where, recall, $\F= \Z$ in the infinite order case and $\F= \Z/q\Z$ in the finite order case. Note that in the latter case it is crucial that $\eta^{q}=\id$ in order to obtain the finite subgroup via the deformation retraction.

To proceed with the proof we specialize to the case when the base space $X$ is a finite simplicial complex. By a general position argument, $X$ can be embedded in $\R^N$, for some large $N$, so that each simplex $\Delta\subset X$ is a linear simplex in $\R^N$. In particular, such embedding yields barycentric coordinates on all simplices. Clearly these coordinates are compatible, in the sense that the restriction of coordinates to a face of a simplex gives the barycentric coordinates in the face. 
We will first consider the infinite order case $\F=\Z$ and then the finite order case $\F= \Z/q\Z$.

\emph{Case 1:} $\F=\Z$.

Let $\mathcal U$ be a finite open cover of $X$ which locally trivializes the bundle, that is, $p^{-1}(U)\to U$, is trivial for all $U\in\mathcal U$. We also choose $\mathcal U$ so that each simplex of $X$ is covered by at least one chart and so that for $U, V\in\mathcal U$ the intersection $U\cap V$ is either empty or connected. Then, because we have reduced the structure group of $p\colon E\to X$ to $\{\eta^k; k\in \F\}$, the transition functions $U\cap V\times M\to U\cap V\times M$ have the form
\begin{equation}
\label{eq_transition}
(x, y)\mapsto (x, \eta^{k(U,V)}(y))
\end{equation}

Recall that by Theorem~\ref{cor:Calegari}, the homeomorphism $\eta$ is isotopic to identity. Denote the isotopy by $\eta^t$, $t\in [0,1]$, $\eta^0=\id$, $\eta^1=\eta$ and let
$$
\eta^{k+t}=\eta^t\circ\eta^k,  k\in\Z, t\in[0,1]
$$
In this way we have embedded the group $\{\eta^k; k\in \Z\}$ into a continuous family $\{\eta^t, t\in\mathbb R\}$. Note however that $\{\eta^t\}$ is not a group.

We will construct the topological trivialization $r\colon E\to M$ using charts from $\mathcal U$. Given a chart $U\in\mathcal U$ the restriction of $r$ to $p^{-1}(U)$ expressed in the $U$-chart will be denoted by $r^U$. By~(\ref{eq_transition}), if $U$ and $V$ overlap then
$$
r^V=r^U\circ\eta^{k(U,V)}
$$
We begin by trivializing over the vertices of $X$. For each vertex $v$ we pick a chart $U\ni v$ and let $r_v^U=\id$. Clearly, changing the chart to a chart $V\ni v$ results in $r_v^V=\eta^{k(U,V)}$. Now let $\Delta=[v_0, v_1, \ldots v_m]\subset X$ be a simplex. For each $x\in \Delta$ let $(t_0, t_1, \ldots t_m)$, $t_j\ge 0, \sum_jt_j=1$, be the barycentric coordinates of $x$. Pick a chart $U\supset \Delta$ and for all $x\in\Delta$ define
$$
r_x^U=\eta^{t_0k_0+t_1k_1+\ldots t_m k_m},
$$
where $k_j$ are given by $r_{v_j}^U=\eta^{k_j}$. To check that $r$ is well defined let $V$ be another chart which covers $\Delta$. Then $r_{v_j}^V=\eta^{k_j+k(U,V)}$ and, indeed,
$$
r_x^V=\eta^{t_0(k_0+k(U,V))+t_1(k_1+k(U,V))+\ldots t_m (k_m+k(U,V))}=\eta^{t_0k_0+t_1k_1+\ldots t_m k_m+k(U,V)}=r_x^U\circ\eta^{k(U,V)}.
$$
It remains to notice that homeomorphisms $r_x^U$ depend continuously on $x\in\Delta$ and, hence, $r_x$ depends continuously on $x\in X$. Thus $r$ provides the posited topological trivialization.

\emph{Case 2:} $\F=\Z/q\Z$, $q\ge 2$.

To prove this case, we will show that we can use the proof of Case 1, but in the universal cover $\tilde E$ of $E$ instead. We will use two things: First a specific knowledge of what the Anosov flow is when $\F=\Z/q\Z$. Second, we will explicitly use the (necessary) assumption that $p\colon E\to X$ is fiberwise homotopically trivial. (In the infinite order case this assumption is also used when reducing the structure group to $F$ \`a la~\cite{FG16}.)

Let $q\colon E\to M$ be a fiber homotopy trivialization, that is, $q$ is a continuous map such that the restrictions to the fibers $q_x\colon M_x\to M$, $x\in X$, are all homotopy equivalences. We can identify the abstract fiber $M$ with a fiber $M_{x_0}$ so that the map induced by $q_{x_0}\colon M_{x_0}\to M$ between the fundamental groups is the identity. Hence, on the fundamental group level, we have that $q_*\colon \pi_1(E)\to \pi_1(M)$ splits the exact sequence
$$
\pi_1(M)\to \pi_1(E)\to \pi_1(X)\to 0
$$
and, thus, $\pi_1(E)\simeq \pi_1(M)\times \pi_1(X)$. Therefore, we can consider the cover $\tilde E\to E$ corresponding to $0\times \pi_1(X)\subset \pi_1(M)\times \pi_1(X)$. The fiber bundle structure lifts and $\tilde E$ is the total space of the bundle
$$
\tilde M\to\tilde E\to X,
$$
where $\tilde M$ is the universal cover of $M$.

Recall that we have reduced the structure group of $E\to X$ to the finite cyclic group $\{\eta^k; k\in \Z/q\Z\}$. Accordingly, the structure group of $\tilde E\to X$ consists of lifts of $\eta^k$ to $\tilde M$ which are equivariantly homotopic to identity. Now we will explain that this structure group is infinite cyclic.

When $\eta^q=\id$, by \cite[Theorem B]{BartFe15}, the flow $\flot$ is orbit equivalent to a $q$-cover of the geodesic flow on the unit tangent bundle of a hyperbolic orbifold or surface. Given the construction of $\eta$ (see, e.g., \cite[Proposition 2.4]{BartFe15} and references therein), in this case, there exists a lift $\wt\eta$ of $\eta$ such that $\wt\eta^q = \gamma$, where $\gamma\in \pi_1(M)$ is the element representing a standard fiber of the Seifert fibration. 
Thus, since the center of $\pi_1(M)$ is generated by $\gamma$, all lifts of $\eta$ which are equivariantly homotopic to identity belong to the infinite cyclic group $\{\tilde\eta^k; k\in \Z\}$. 

Then, the proof proceeds as in Case 1, by constructing a fiberwise topological trivialization of $\tilde E\to X$ but using $\tilde\eta$ instead of $\eta$. Since the construction of such a trivialization is equivariant with respect to the action of $\pi_1(M)$ on $\tilde E$, it yields the posited topological trivialization of $E\to X$. This concludes the proof in Case 2.

%

Finally, it remains to address the case when $X$ is a closed manifold (there are many topological manifolds which do not admit triangulations). Recall that every closed manifold is an ANR (Absolute Neighborhood Retract) and hence is homotopy equivalent to a finite simplicial complex $\bar X$~\cite{W77}. Denote by $\alpha\colon \bar X\to X$ a homotopy equivalence and by $\beta\colon X\to \bar X$ its homotopy inverse. We can pull back $p\colon E\to X$ to a bundle $\alpha^*E\to\bar X$ using $\alpha$. This pull-back bundle $\alpha^*E\to\bar X$ is naturally equipped with a fiberwise Anosov flow by pulling back the fiberwise Anosov flow on $E\to X$. Therefore, by the above proof $\alpha^*E\to\bar X$ is topologically trivial. 

Now we can pull back using $\beta$ to obtain a bundle $\beta^*\alpha^*E\to X$ which is obviously topologically trivial. But this bundle is isomorphic to the original bundle $E\to X$ because $\alpha\circ\beta$ is homotopic to identity (and the isomorphism class of the bundle is determined by the homotopy class of the map). This completes the proof in the case when $X$ is a closed manifold as well.

\bibliographystyle{amsalpha} 
\bibliography{fiberwise_anosov}
\end{document}